\documentclass[a4paper, final]{amsart}

\usepackage[utf8]{inputenc}
\usepackage[T1]{fontenc}
\usepackage[english]{babel}
\usepackage{amsmath,amsthm}
\usepackage{ae}
\usepackage{icomma}
\usepackage{units}
\usepackage{color}
\usepackage{graphicx}
\usepackage{bbm}
\usepackage{caption}
\usepackage{array}
\usepackage[hmarginratio=1:1]{geometry}
\usepackage[hyphens]{url}
\usepackage[pdfpagelabels=false]{hyperref}
\usepackage{mathrsfs}
\usepackage{amssymb}
\usepackage{placeins} 
\usepackage{float} 
\usepackage[nodayofweek]{datetime}
\usepackage{enumitem}
\usepackage[square, numbers, sort]{natbib}

\renewenvironment{thebibliography}[1]{%
   \begin{oldthebibliography}{#1}%
     \setlength{\itemsep}{1pt}%
}%
{%
   \end{oldthebibliography}%
}

\usepackage{mathtools}
\mathtoolsset{showonlyrefs=true}

\newcommand{\limplus}{{\mathchoice{\vcenter{\hbox{$\scriptstyle +$}}}
  {\vcenter{\hbox{$\scriptstyle +$}}}
  {\vcenter{\hbox{$\scriptscriptstyle +$}}}
  {\vcenter{\hbox{$\scriptscriptstyle +$}}}
}}

\newcommand{\grad}{\ensuremath{\nabla}}
\newcommand{\R}{\ensuremath{\mathbb{R}}}
\DeclareMathOperator{\dist}{\textnormal{dist}}
\renewcommand{\S}{\ensuremath{\mathbb{S}}}
\newcommand{\Heis}{\ensuremath{\mathbb{H}}}
\newcommand{\Pip}{\ensuremath{{\Pi_{\nu, d}^\limplus}}}
\newcommand{\gradH}{\ensuremath{{\grad_{\Heis^n}}}}
\newcommand{\HorizSubspace}[1]{\ensuremath{\mathcal{H}_{#1}}}
\newcommand{\normal}{\ensuremath{n}}
\newcommand{\sgn}{\ensuremath{\textnormal{sgn}}}
\renewcommand{\div}{\textnormal{div}}

\newcommand{\ptS}{\hspace{-1pt} }

\newtheorem{theorem}{Theorem}[section]

\newtheorem{corollary}[theorem]{Corollary}
\numberwithin{theorem}{section}
\numberwithin{definition}{section}

\newcommand{\thmref}[1]{Theorem~\ref{#1}}

\begin{document}

\pagenumbering{arabic}

\title[Geometric Hardy inequalities on convex domains in $\Heis^n$]{Geometric Hardy inequalities for the sub-elliptic Laplacian on convex domains in the Heisenberg group}

\author{Simon Larson}
\address{Department of Mathematics, Royal Institute of Technology, SE-10044 Stockholm, Sweden}

\email{simla@math.kth.se}


\date{\today.}

\begin{abstract}
We prove geometric $L^p$ versions of Hardy's inequality for the sub-elliptic Laplacian on convex domains $\Omega$ in the Heisenberg group $\mathbb{H}^n$, where convex is meant in the Euclidean sense. When $p=2$ and $\Omega$ is the half-space given by $\langle \xi, \nu\rangle > d$ this generalizes an inequality previously obtained by Luan and Yang. For such $p$ and $\Omega$ the inequality is sharp and takes the form
\begin{equation}
	 \int_\Omega |\nabla_{\mathbb{H}^n}u|^2 \, d\xi \geq \frac{1}{4}\int_{\Omega} \sum_{i=1}^n\frac{\langle X_i(\xi), \nu\rangle^2+\langle Y_i(\xi), \nu\rangle^2}{\dist(\xi, \partial \Omega)^2}|u|^2\, d\xi, 
\end{equation}
where $\dist(\, \cdot\,, \partial \Omega)$ denotes the Euclidean distance from $\partial \Omega$.
\end{abstract}

\maketitle


\section{Introduction}
In~\cite{LuanYang} Luan and Yang prove the Hardy inequality
\begin{equation}\label{eq:LuanYangHardy}
	\int_{\Heis^n_\limplus}|\grad_{\Heis^n} u|^2 \, d\xi \geq \int_{\Heis_\limplus^n}\frac{|x|^2+|y|^2}{t^2}|u|^2\, d\xi,
\end{equation}
where an element $\xi\in \Heis^n$ is written as $\xi=(x, y, t)$, with $x, y\in \R^n$ and $t\in \R$, and $\Heis^n_\limplus:=\{\xi\in\Heis^n : t>0\}$. In this paper we provide a different proof of this inequality, generalize it to any half-space of $\Heis^n$ and use it to obtain a weighted geometric Hardy inequality on a convex domain $\Omega$, where convex is meant in the Euclidean sense. The weight that appears in our results is in some sense a natural sub-elliptic weighting of the Euclidean distance and is closely related to distances studied in~\cite{PrandiRizziSeri, RuszkowskiWeidl}.

We begin with a short introduction providing the basic definitions, notation and background necessary for the sequel.


The $n$-dimensional \emph{Heisenberg group}, which we denote by $\Heis^n$, may be described as the set $\R^{2n+1}$ equipped with the group law
\begin{equation}
	\hat\xi \circ \tilde\xi := (\hat x+\tilde x, \hat y+ \tilde y, \hat t+\tilde t + 2 \sum_{i=1}^n (\tilde x_i \hat y_i - \hat x_i \tilde y_i)), 
\end{equation}
where we use the notation $\xi = (x_1, \dots, x_n, y_1, \dots, y_n, t)=(x, y, t)\in \R^{2n+1}$. The inverse element of $\xi$, with respect to the group law, is denoted by $\xi^{-1}$ and we note that $\xi^{-1}=-\xi$. The group law induces the following dilation operation
\begin{equation}
	\delta_\lambda(\xi) := (\lambda x, \lambda y, \lambda^2 t)\quad \textnormal{for } \lambda>0.
\end{equation}

The Lie algebra of left-invariant vector fields on $\Heis^n$ is spanned by 
\begin{align}
	X_i := \frac{\partial }{\partial x_i}+2 y_i \frac{\partial }{\partial t} \quad \textrm{and} \quad Y_i := \frac{\partial }{\partial y_i}-2 x_i \frac{\partial }{\partial t}, 
\end{align}
for $1\leq i\leq n$, together with their commutators. The only non-zero commutators are 
\[
	[X_i, Y_i]=-4 \frac{\partial }{\partial t}.
\]

We also define the associated gradient $\gradH\!:= (X_1, \dots, X_n, Y_1, \dots, Y_n)$ and the \emph{Heisenberg Laplacian} $\Delta_{\Heis^n}$ on $\Heis^n$, formally given by $\Delta_{\Heis^n}\!:= \sum_{i=1}^n X_i^2+Y_i^2$. The collection of vector fields $\{X_i, Y_i : 1\leq i\leq n\}$ satisfies the H\"ormander finite rank condition:
\begin{equation}
	\textrm{Rank } \textrm{Lie}[X_1, \dots, X_n, Y_1, \dots, Y_n]=2n+1.
\end{equation}
Thus the Heisenberg Laplacian is a second order hypoelliptic differential operator~\cite{Hormander}.

We call a Lipschitz curve $\gamma\colon I\subset \R \to \Heis^n$ \emph{horizontal} if its tangent at almost every $\tau \in I$ is spanned by the $X_i$ and $Y_i$, that is, for a.e.\ $\tau\in I$ there exist $a, b \in \R^n$ such that
\begin{equation}
	\gamma'(\tau)=\sum_{i=1}^n a_iX_i(\gamma(\tau))+b_iY_i(\gamma(\tau)).
\end{equation}
We denote the set of all horizontal curves $\gamma\colon I \to \Heis^n$ by $\mathcal{S}_{\Heis^n}(I)$ and for a given $\gamma\in \mathcal{S}_{\Heis^n}(I)$ we define its length $l(\gamma)$ as
\begin{equation}
	l(\gamma):=\int_I \bigl(|a(\tau)|^2+|b(\tau)|^2\bigr)^{1/2}\, d\tau.
\end{equation}
By the accessibility theorem of Chow and Rashevsky any pair of points $\xi_0, \xi_1 \in \Omega$ where $\Omega$ is an open connected subset of $\Heis^n$ can be joined by a horizontal curve $\gamma\colon [0, 1]\to \Heis^n$ of finite length (see~\cite{MR0001880, Rashevsky}).

The \emph{Carnot--Carath\'eodory distance} $\delta_{cc}$ on $\Heis^{n}$ is defined as
\begin{equation}\label{eq:DefCCdistance}
	\delta_{cc}(\xi_0, \xi_1) := \inf \{l(\gamma) : \gamma\in \mathcal{S}_{\Heis^n}([0, 1]), \gamma(0)=\xi_0, \gamma(1)=\xi_1\}.
\end{equation}

The Carnot--Carath\'eodory distance is not the only distance that has a natural connection to $\Heis^n$. A second distance that arises naturally when considering the fundamental solution of $\Delta_{\Heis^n}$ is the \emph{Kaplan distance} (see~\cite{MR1070830}):
\begin{equation}
	\delta_{K}(\xi_0, \xi_1):= \rho(\xi_1^{-1}\!\circ \xi_0), 
\end{equation}
where $\rho$ is the \emph{Kaplan gauge} on $\Heis^n$ defined by
\begin{equation}
	\rho(\xi):=\bigl((|x|^2+|y|^2)^2+ 4 t^2\bigr)^{1/4}.
\end{equation}
It turns out that the two distance functions above are bi-Lipschitz equivalent, that is, there exists a constant $C>0$ such that for all $\xi^0, \xi^1\in \Heis^n$ we have that
\begin{equation}
	C^{-1} \delta_K(\xi_0, \xi_1)\leq \delta_{cc}(\xi_0, \xi_1) \leq C \delta_K (\xi_0, \xi_1).
\end{equation}

Let $\mathcal{M}\subset \Heis^n$ be a $2n$-dimensional $C^1$ manifold. We call a point $\xi_0 \in \mathcal{M}$ a \emph{characteristic point} of $\mathcal{M}$ if the tangent space $T_{\xi_0}\mathcal{M}$ is spanned by $\{ X_i(\xi_0),, Y_i(\xi_0) : 1\leq i \leq n\}$.

Even though both $\delta_{cc}$ and $\delta_{K}$ appear naturally when considering the geometric structure of $\Heis^n$ these distances can be rather difficult to work with, see for instance the work of Arcozzi and Ferrari \cite{ArcozziFerrari1, ArcozziFerrari2}. Extra difficulties arise when studying the behaviour of the distance to a hypersurface $\mathcal{M}$ close to one of its characteristic points.

In what follows we will be interested in inequalities of the form
\begin{equation}\label{eq:GeneralHardyDistance}
	\int_\Omega |\gradH u(\xi)|^p \, d\xi \geq C \int_\Omega \frac{|u(\xi)|^p}{\rho(\xi, \partial \Omega)^p} \, d\xi, 
\end{equation}
where $\Omega \subset \Heis^n$, $p\geq 2$ and $\rho$ is some, possibly weighted, distance from $\xi$ to the boundary of $\Omega$. 
In the Euclidean setting, with $\gradH$ replaced by the usual gradient and $\rho$ by the Euclidean distance, such inequalities have a long history and wide range of applications (see, for instance,~\cite{MR990239, MR2777530, BalEvLew}).

In the setting of the Heisenberg group results of this kind have been obtained through methods based on sub-elliptic capacity and Fefferman--Phong inequalities. In~\cite{DanielliGarofaloPhuc} the authors provide sharp conditions on the triple $\Omega, \rho$ and $C$ for the validity of~\eqref{eq:GeneralHardyDistance}. However, the results obtained in~\cite{DanielliGarofaloPhuc} are of a rather non-explicit nature and what they say in a specific setting is not very approachable.  

One of the obstacles in proving inequalities of the form~\eqref{eq:GeneralHardyDistance} on domains in $\Heis^n$ is that the natural distances on $\Heis^n$ ($\delta_{cc}, \delta_K$) are rather difficult to work with. If $\Omega=\Heis^n\setminus\{\xi_0\}$ and $\rho$ is the Carnot--Carath\'eodory or Kaplan distance to the point $\xi_0\in\Heis^n$ inequalities of this form have been studied in a several articles (see, for instance,~\cite{Ambrosio, MR1070830, MR2447488, RuzhanskySuragan}). However, when $\partial \Omega$ is a more complicated set the problem becomes more difficult. Results concerning the behaviour of the distance from sets and a detailed analysis of the problems arising can be found in work by Arcozzi and Ferrari~\cite{ArcozziFerrari1, ArcozziFerrari2}. 

In this article we begin by generalizing~\eqref{eq:LuanYangHardy} to the case where $\Omega$ is an arbitrary half-space of~$\Heis^n$. The proof given here differs from that given in~\cite{LuanYang} and contains their result as a special case. Moreover, from our proof of~\eqref{eq:LuanYangHardy} and the corresponding generalizations we are able to apply a standard argument and find $L^p$ versions of the inequalities.

In Section~\ref{sec:IneqConvexSet}, we combine our inequalities for half-spaces with a method used by Avkhadiev in the Euclidean setting~\cite{Avkhadiev1} to obtain an inequality of the form~\eqref{eq:GeneralHardyDistance} for convex domains in $\Heis^n$, here convex is meant in the Euclidean sense, and with $\rho$ being a weighted Euclidean distance. More specifically we have that
\begin{equation}
 	\frac{1}{\rho(\xi)^p}=\sum_{i=1}^n\frac{|\langle X_i(\xi), \nu(\xi) \rangle|^p + |\langle Y_i(\xi), \nu(\xi) \rangle|^p}{\dist(\xi, \partial \Omega)^p}, 
\end{equation}
where $p\geq 2$, $\dist(\, \cdot\,, \partial\Omega)$ denotes the Euclidean distance to the boundary of $\Omega$ and $\nu(\xi)\in \S^{2n}$ is such that $\xi + \dist(\xi, \partial \Omega)\nu(\xi) \in \partial \Omega$.


\section{Hardy inequalities on half-spaces of \texorpdfstring{$\Heis^n$}{Hn}}

For $\nu\in \S^{2n}$ and $d\in \R$ let $\Pi_{\nu, d}$ be the hyperplane in $\Heis^n$ defined by the equation $\langle \xi, \nu\rangle =d$. Correspondingly, let $\Pip$ be the half-space of $\Heis^n$ where $\langle \xi, \nu\rangle>d$.
\begin{theorem}\label{thm:IneqHalfspace}
 	Let $u\in C_0^\infty(\Pip)$. Then the following inequality holds
 	\begin{equation}\label{eq:IneqHalfspace}
 		\int_{\Pip}|\grad_{\Heis^n}u|^2\, d\xi \geq \frac{1}{4}\int_{\Pip}\sum_{i=1}^n\frac{\langle X_i(\xi), \nu\rangle^2+\langle Y_i(\xi), \nu \rangle^2}{\dist(\xi, \partial\Pip)^2}|u|^2\, d\xi.
 	\end{equation}
\end{theorem}

By choosing $\nu$ to be the unit vector in the $t$ direction and $d$ to be zero the above theorem reduces to~\eqref{eq:LuanYangHardy}. 

In the case of a half-space it was pointed out to us by Ruszkowski that outside a certain cone the weighted distance appearing in \thmref{thm:IneqHalfspace} is comparable to the Carnot--Carath\'eodory distance on the Heisenberg group. In fact, the weighted distance coincides with a reduced version of the Carnot--Carath\'eodory distance~\cite{RuszkowskiWeidl,PrandiRizziSeri}, namely 
\begin{equation}
w(\xi, \partial \Omega):= \inf\{\delta_{cc}(\xi, \hat\xi) : \hat\xi \in \partial\Omega \cap \textrm{Span}(X_i(\xi),  Y_i(\xi): 1\leq i\leq n)\}.
\end{equation}
For results concerning this reduced distance and Hardy inequalities closely related to those obtained here we refer to~\cite{PrandiRizziSeri} and an upcoming article by Ruszkowski and Weidl~\cite{RuszkowskiWeidl}.

We proceed by providing a factorization-type proof of the above theorem and also sketch how to obtain the same statement from~\eqref{eq:LuanYangHardy} through a simple translation argument. The second argument has the slight advantage that it gives a geometric interpretation of the weight appearing in the inequality, but most importantly it simplifies the proof that~\eqref{eq:IneqHalfspace} is sharp. However, later in the article we will need the calculations performed in our first proof.
 
\begin{proof}[Proof of~\thmref{thm:IneqHalfspace}]
	The inequality is obtained by a simple factorization argument and the optimization of a parameter $\alpha$. For $u\in C^\infty_0(\Pip)$ and any $V\in H^1(\Pip; \R^{2n})$ with components $(V_1, \dots, V_{2n})$ we have that
	\begin{align}
	 	0 & \leq \int_\Pip |(\gradH+\alpha V)u|^2\, d\xi \\[3pt]
	 	& = 
	 	\sum_{i=1}^n \int_\Pip \bigl(|(X_i+\alpha V_{i})u|^2+|(Y_i+\alpha V_{n\limplus i})u|^2\bigr) d\xi\\[3pt]
	 &=
	 	\sum_{i=1}^n \int_\Pip \bigl(|X_i u|^2+|Y_iu|^2 -\alpha |u|^2 (X_i(V_i) + Y_i(V_{n\limplus i}))
	 	+\alpha^2 |u|^2(V_i^2+V_{n\limplus i}^2)\bigr) d\xi,
	\end{align}
	where the last equality is obtained by partial integration and the fact that $u$ has compact support. 

	Rearranging the terms one finds the following inequality
	\begin{equation}\label{eq:VectorPotentialV}
	 	\int_\Pip |\gradH u|^2 \geq \sum_{i=1}^n \int_\Pip \alpha |u|^2 \bigl(X_i(V_i) + Y_i(V_{n\limplus i})-\alpha(V_i^2+V_{n\limplus i}^2)\bigr)\, d\xi.
	\end{equation}
	We now choose the components of $V$ as 
	\begin{align}
	 	V_i(\xi) &= \frac{\langle X_i(\xi), \nu\rangle}{\dist(\xi, \partial \Pip)} = \frac{\langle X_i(\xi), \nu\rangle}{\langle \xi, \nu\rangle -d}\\
	 	V_{n\limplus i}(\xi) &= \frac{\langle Y_i(\xi), \nu\rangle}{\dist(\xi, \partial \Pip)} = \frac{\langle Y_i(\xi), \nu\rangle}{\langle \xi, \nu\rangle -d}.
	\end{align}
	A simple calculation gives us that
	\begin{align}
	 	X_i(V_i)(\xi) = 
	 	-\frac{\langle X_i(\xi), \nu\rangle^2}{\dist(\xi, \partial\Pip)^2} 
	 \quad \textrm{and}\quad
	 	Y_i(V_{n\limplus i})(\xi)=-\frac{\langle Y_i(\xi), \nu\rangle^2}{\dist(\xi, \partial\Pip)^2}.
	\end{align}
	Inserting into equation \eqref{eq:VectorPotentialV} we find that
	\begin{align}
	 	\int_\Pip |\gradH u|^2 \geq  -\alpha(1+\alpha) \int_\Pip  \sum_{i=1}^n\frac{\langle X_i(\xi), \nu\rangle^2+\langle Y_i(\xi), \nu\rangle^2}{\dist(\xi, \partial\Pip)^2} |u|^2\, d\xi.
	\end{align}
	Choosing $\alpha$ to maximize $-\alpha(1+\alpha)$ completes the proof. 
\end{proof}

As mentioned above the theorem admits a second proof through a translation argument, and we proceed by sketching this alternative proof. The main reason for including this is that it reduces the proof of sharpness of~\eqref{eq:IneqHalfspace} to considering a given half-space, but it also gives a second geometric interpretation of the weight appearing in the theorem.

We sketch the proof only in the case of $\Heis^1$. The ideas translate without change to higher dimension but the geometry of the argument is more transparent in the case $n=1$. For simplicity we will also only deal with the case $d=0$, i.e.\ a plane passing through the origin.

For $\nu=(\nu_x, \nu_y, \nu_t)\in \S^2$, let $\Pi_{\nu}$ be the plane in $\Heis^1$ defined by the equation $\langle \xi, \nu\rangle=0$. If $\nu_t\neq 0$ we can find a $\xi^0\in \Pi_\nu$ such that $\xi^0$ is a characteristic point of $\Pi_\nu$. This reduces to solving the following system of equations:
\begin{equation}
 	\left\{\begin{matrix}
 		\langle X(\xi), \nu\rangle =0,\\
 		\langle Y(\xi), \nu\rangle =0,\\
 		\langle \xi, \nu\rangle =0.
	\end{matrix}\right.
\end{equation}
From the first two equations we find that
 \begin{align}
	x = \frac{\nu_y}{2\nu_t}\quad \textnormal{and} \quad
	y = -\frac{\nu_x}{2\nu_t}
\end{align}
which combined with the third equation gives the solution
\begin{equation}
	\xi^0=\frac{1}{2\nu_t}\left(\begin{matrix}\nu_y\\-\nu_x\\0\end{matrix}\right).
\end{equation}
Applying a change of variables given by left translation by $-\xi^0$ and using the left-invariance of $X$ and $Y$ reduces the left-hand side of inequality~\eqref{eq:IneqHalfspace} to the case of $\Pi_\nu=\{\xi\in\Heis^1 : t=0\}$. Thus we may apply~\eqref{eq:LuanYangHardy}. Changing back variables the right-hand side becomes, after some algebraic manipulations, the desired expression. For arbitrary non-vertical planes ($\nu_t\neq 0$) the argument goes through without any substantial change. For vertical planes ($\nu_t=0$) the result can be found through a simple limiting process.

In the case treated by Luan and Yang the term $|x|^2+|y|^2$ has a clear interpretation as the square of the Euclidean distance to the centre of $\Heis^1$, which here actually coincides with the Carnot--Carath\'eodory distance. This is precisely the distance from $\xi$ to the subspace of $\Heis^1$ consisting of points where $\nu$ is orthogonal to both $X$ and $Y$. 

The translation argument above then provides the interpretation that the weight corresponds to the Carnot--Carath\'eodory distance from the subspace where the $X_i$ and $Y_i$ span a hyperplane which is parallel to $\Pi_\nu$, multiplied by a factor corresponding to how "tilted" $\Pi_\nu$ is. If we let $\HorizSubspace{\nu}$ denote the subspace of $\Heis^1$ given by $\{\xi\in\Heis^1 : X(\xi)\perp \nu, Y(\xi)\perp \nu\}$ one finds that
\begin{equation}
	\frac{\langle X(\xi), \nu\rangle^2+\langle Y(\xi), \nu\rangle^2}{4} = \nu_t^2 \delta_{cc}^2(\xi, \HorizSubspace{\nu}).
\end{equation}
One should note that this weight behaves well with respect to both $\xi$ and $\nu$. In particular, if we let $\nu_t$ tend to zero this converges to $1/4$. As this leads to a rather surprising invariance of the Hardy inequality with respect to any choice of vertical plane we state this as a corollary.
\begin{corollary}
	Let $u\in C_0^\infty (\Pip)$ with $\nu_t=0$. Then the following inequality holds
	\begin{equation}
		\int_{\Pip} |\gradH u|^2\, d\xi \geq \frac{1}{4}\int_{\Pip} \frac{|u|^2}{\dist(\xi, \partial\Pip)^2}\, d\xi.
	\end{equation}
\end{corollary}
With the above translation argument in hand we see that to prove the sharpness~\eqref{eq:IneqHalfspace} it suffices to consider a given pair of $\nu\in \S^{2n}$ and $d\in \R$. Thus we restrict our attention to the case $\Pip = \{ (x, y, t)\in\Heis^n : t \geq 0\}$, that is, $\nu=(0, \dots, 0, 1)$ and $d=0$. We will also make use of the identity
\begin{equation}
	\gradH u = \grad' u + 2 \Lambda \xi' \frac{\partial u}{\partial t}, 
\end{equation}
where $\xi'=(x, y)$, $\grad'$ denotes the gradient in $\R^{2n}$ acting in the $\xi'$ variables and $\Lambda$ is the skew symmetric matrix
\begin{equation}
	\left(\begin{matrix}
		0  & I_n  \\
		- I_n & 0
	\end{matrix}\right).
\end{equation}
It follows that
\begin{align}
	|\gradH u |^2 &= 
	\langle \grad'u + 2 \Lambda \xi' \frac{\partial u}{\partial t}, \grad'u + 2 \Lambda \xi' \frac{\partial u}{\partial t}\rangle \\
	&=|\grad' u |^2 + 4 \frac{\partial u}{\partial t} \langle \Lambda \xi', \grad'u \rangle + 4 |\Lambda \xi'|^2 \Bigl|\frac{\partial u}{\partial t}\Bigr|^2\\[2pt]
	&=|\grad' u |^2 + 4 \frac{\partial u}{\partial t} \langle \Lambda \xi', \grad'u \rangle + 4 |\xi'|^2 \Bigl|\frac{\partial u}{\partial t}\Bigr|^2.
\end{align}

The sharpness of~\eqref{eq:IneqHalfspace} now follows from a straightforward variational argument with the ansatz $u(x, y, t)=w(t)\phi(x, y)$. We argue as follows:
\begin{align}
	\inf_{u\in C_0^\infty(\Pip)} &\frac{\int_\Pip |\gradH u |^2\, d\xi}{\int_\Pip \frac{|x|^2+|y|^2}{t^2}|u|^2\, d\xi} 
	\leq  \inf_{\substack{\ \phi \in C_0^\infty(\R^{2n})\\ w \in C_0^\infty(\R_\limplus)}} \frac{\int_\Pip|\gradH (\phi w) |^2\, d\xi}{\int_\Pip \frac{|x|^2+|y|^2}{t^2}|\phi w|^2\, d\xi}\\[5pt]
	&\quad =
	\inf_{\substack{\ \phi \in C_0^\infty(\R^{2n})\\ w \in C_0^\infty(\R_\limplus)}} \frac{\int_{\R_\limplus}\int_{\R^{2n}} \bigl(|w\grad \phi |^2 + 4 w w' \phi \langle \Lambda \xi', \grad \phi \rangle + 4 |\xi'|^2 |\phi w'|^2 \bigr) d\xi'\, dt }{\int_{\R_\limplus}\int_{\R^{2n}} \frac{|w|^2}{t^2}|\xi'|^2|\phi|^2\, d\xi' dt}\\[5pt]
	&\quad=
	\inf_{\substack{\ \phi \in C_0^\infty(\R^{2n})\\ w \in C_0^\infty(\R_\limplus)}} \Biggl[
	\frac{\int_{\R_\limplus} |w|^2 \, dt}{\int_{\R_\limplus}\frac{|w|^2}{t^2} \, dt} \cdot
	\frac{\int_{\R^{2n}} |\grad \phi |^2 \, d\xi'}
	{ \int_{\R^{2n}}|\xi'|^2 |\phi|^2\, d\xi'}
	+ 
	4 \frac{\int_{\R_\limplus}w w'\, dt}{\int_{\R_\limplus}\frac{|w|^2}{t^2} \, dt}\cdot
	\frac{\int_{\R^{2n}} \phi \langle \Lambda \xi', \grad \phi \rangle \, d\xi'}
	{ \int_{\R^{2n}}|\xi'|^2 |\phi|^2\, d\xi'}\\[5pt]
	&\quad\hphantom{=\inf_{\substack{\ \phi \in C_0^\infty(\R^{2n})\\ w \in C_0^\infty(\R_\limplus)}}}
	+ 
	4 \frac{\int_{\R_\limplus}|w'|^2\, dt }{\int_{\R_\limplus}\frac{|w|^2}{t^2} \, dt} \cdot
	\frac{ \int_{\R^{2n}}|\xi'|^2 |\phi|^2 \, d\xi'}
	{ \int_{\R^{2n}}|\xi'|^2 |\phi|^2\, d\xi'}\Biggr]\\[5pt]
	&\quad=
	\inf_{\substack{\ \phi \in C_0^\infty(\R^{2n})\\ w \in C_0^\infty(\R_\limplus)}} \Biggl[
	\frac{\int_{\R_\limplus} |w|^2 \, dt}{\int_{\R_\limplus}\frac{|w|^2}{t^2} \, dt} \cdot
	\frac{\int_{\R^{2n}} |\grad \phi |^2 \, d\xi'}
	{ \int_{\R^{2n}}|\xi'|^2 |\phi|^2\, d\xi'} 
	+ 
	4 \frac{\int_{\R_\limplus}|w'|^2\, dt }{\int_{\R_\limplus}\frac{|w|^2}{t^2} \, dt}\Biggr]= 1, 
\end{align}
where we used that $2 \int_{\R_\limplus} w w' \, dt = \int_{\R_\limplus} (w^2)' \, dt =0$, the sharp Hardy inequality on $\R_\limplus$ and that
\begin{equation}
	\inf_{\phi \in C_0^\infty(\R^{2n})} \frac{\int_{\R^{2n}} |\grad \phi |^2 \, d\xi'}
	{ \int_{\R^{2n}}|\xi'|^2 |\phi|^2\, d\xi'} =0, 
\end{equation}
which follows by a simple dimensionality argument. Hence the constant appearing in \thmref{thm:IneqHalfspace} is sharp. For vertical planes, i.e.\ the ones that can not be reached by translation, the sharpness follows by a limiting procedure but can also be found by an almost identical variational argument but considering a slightly different quotient.


\subsection{An \texorpdfstring{$L^p$}{Lp} Hardy inequality on a half-space of \texorpdfstring{$\Heis^n$}{Hn}}

We may use standard techniques to generalize the proof of the previous theorem to construct $L^p$ Hardy inequalities for any $p\geq 2$. We summarize the results in the following theorem.

\begin{theorem}\label{thm:IneqHalfspaceLP}
	For $\nu\in \S^{2n}$ and $d\in \R$ let $\Pip$ be the half-space described previously. Then for $p\geq 2$ and $u\in C_0^\infty(\Pip)$ the following inequality holds
	\begin{equation}\label{eq:IneqHalfspaceLP}
		\int_\Pip |\gradH u|^p\, d\xi \geq \Bigl(\frac{p-1}{p}\Bigr)^p \int_\Pip \sum_{i=1}^n \frac{|\langle X_i(\xi), \nu\rangle|^p+|\langle Y_i(\xi), \nu\rangle|^p}{\dist(\xi, \partial\Pip)^p}|u|^p\, d\xi.
	\end{equation}
	Moreover, the constant in the inequality is sharp.
\end{theorem}

The above $L^p$ version of our Hardy inequality on a half-space is perhaps not the most natural generalization of \thmref{thm:IneqHalfspace}; a more natural weight in the right-hand side of~\eqref{eq:IneqHalfspaceLP} would be
\begin{equation}
	\frac{\bigl(\sum_{i=1}^n \langle X_i(\xi), \nu\rangle^2+\langle Y_i(\xi), \nu\rangle^2\bigr)^{p/2}}{\dist(\xi, \partial\Pip)^p}.
\end{equation}
However, by Jensen's inequality it is easy to see that \thmref{thm:IneqHalfspaceLP} implies a Hardy inequality with the above weight but with a worse constant. We believe that such an inequality should hold with the same constant as in~\eqref{eq:IneqHalfspaceLP}, namely $\bigl(\ptS\frac{p-1}{p}\ptS\bigr)^p$ (which is the sharp constant also for the Euclidean counterpart), but so far we are not able to prove this. 

\begin{proof}[Proof of \thmref{thm:IneqHalfspaceLP}]
	The proof of the theorem is very similar to the proof presented above for the $L^2$ case. By the divergence theorem we have for $g\in H^1(\Pip)$ and $V\in C^\infty(\Pip; \R^{2n+1})$ that
	\begin{align}\label{eq:DivTheoremHalfspaceLp}
	 	\int_{\Pip} \div(gV)|u|^p\, d\xi 
	 	& = 
	 	-p \int_{\Pip} g \langle V, \grad u\rangle\, \sgn(u)|u|^{p-1}\, d\xi. 
	\end{align} 
	Here and in what follows $\div$ and $\grad$ denote the usual divergence and gradient in $\R^{2n+1}$.
	By H\"older's and Young's inequalities we have that
	\begin{align}
	 	-p \int_{\Pip} g \langle V, \grad u\rangle\, \sgn(u) |u|^{p-1}\, d\xi 
	 	&\leq 
	 	p \Bigl(\int_{\Pip}|\langle V, \grad u\rangle|^p\, d\xi\Bigr)^{1/p}\Bigl(\int_{\Pip}|g|^{p/(p-1)}|u|^p\, d\xi\Bigr)^{(p-1)/p}\\
	 	&\leq 
	 	\int_{\Pip}|\langle V, \grad u\rangle|^p\, d\xi+(p-1)\int_{\Pip}|g|^{p/(p-1)}|u|^p\, d\xi.
	\end{align} 
	Inserting this into~\eqref{eq:IneqHalfspaceLP} and rearranging the terms we obtain
	\begin{align}
		\int_{\Pip}\bigl(\div (gV)-(p-1)|g|^{p/(p-1)}\bigr)|u|^p\, d\xi
		& \leq
		\int_{\Pip}|\langle V, \grad u\rangle|^p\, d\xi.
	\end{align}
	Choose $V=X_i$ and let
	\begin{equation}
		g=\alpha\, \sgn(\langle X_i(\xi), \nu\rangle)\biggl(\frac{|\langle X_i(\xi), \nu\rangle|}{\dist(\xi, \partial \Pip)}\biggr)^{p-1}.
	\end{equation}
	By the same calculations as earlier we obtain that
	\begin{align}\label{eq:XcompontentIneq}
		C(\alpha, p)\int_{\Pip}\frac{|\langle X_i(\xi), \nu \rangle|^p}{\dist(\xi, \partial \Pip)^p}|u|^p\, d\xi
		& \leq
		\int_{\Pip}|X_iu|^p\, d\xi
	\end{align}
	where 
	\begin{equation}
		C(\alpha, p)=-(p-1)(\alpha+|\alpha|^{p/(p-1)}).
	\end{equation}
	Maximizing this constant in $\alpha$ we find that
	\begin{equation}
	 	C(\alpha, p)\leq \Bigl(\frac{p-1}{p}\Bigr)^p, 
	\end{equation} 
	where the maximum is attained at
	\begin{equation}
		\alpha=-\Bigl(\frac{p-1}{p}\Bigr)^{p-1}.
	\end{equation}
	
	By an almost identical calculation but with $V=Y_i$ and
	\begin{equation}
		g=\alpha\, \sgn(\langle Y_i(\xi), \nu\rangle)\biggl(\frac{|\langle Y_i(\xi), \nu\rangle|}{\dist(\xi, \partial\Pip)}\biggr)^{p-1}
	\end{equation}
	one finds that
	\begin{equation}\label{eq:YcomponentIneq}
		\Bigl(\frac{p-1}{p}\Bigr)^p \int_{\Pip}\frac{|\langle Y_i(\xi), \nu \rangle|^p}{\dist(\xi, \partial\Pip)^p}|u|^p\, d\xi
		 \leq
		\int_{\Pip}|Y_iu|^p\, d\xi.
	\end{equation}

	Adding the two inequalities and summing over $i=1, \ldots, n$ we find that
	\begin{equation}
		\Bigl(\frac{p-1}{p}\Bigr)^p \sum_{i=1}^n\int_{\Pip}\frac{|\langle X_i(\xi), \nu\rangle|^p+|\langle Y_i(\xi), \nu \rangle|^p}{\dist(\xi, \partial\Pip)^p}|u|^p\, d\xi
		 \leq
		\int_{\Pip}\sum_{i=1}^n\bigl(|X_iu|^p+|Y_iu|^p\bigr)\, d\xi.
	\end{equation}
	Since $p\geq 2$, the function $\varphi:x \mapsto x^{p/2}$ is superadditive, and therefore
	\begin{equation}
		\sum_{i=1}^n \bigl(|X_iu|^p+|Y_iu|^p\bigr) = \sum_{i=1}^n \bigl(|X_iu|^2\bigr)^{p/2}+\bigl(|Y_iu|^2\bigr)^{p/2} \leq \Bigl(\sum_{i=1}^n |X_iu|^2+|Y_iu|^2\Bigr)^{p/2} = |\gradH u|^p.
	\end{equation}
	Inserting this into the above we get~\eqref{eq:IneqHalfspaceLP}.

	By a similar argument as in the case $p=2$ we can prove that the constant is sharp in the sense that if it were replaced by a larger constant we could choose $\Pip$ such that the inequality fails. To achieve this we wish to find an upper bound for the quantity
	\begin{align}
	  	\inf_{u\in C^{\infty}_0(\Pip)} \frac{\int_\Pip |\gradH u|^p\, d\xi}{\int_\Pip \sum_{i=1}^n \frac{|\langle X_i(\xi), \nu\rangle|^p+|\langle Y_i(\xi), \nu\rangle|^p}{\dist(\xi, \partial\Pip)^p}|u|^p\, d\xi}.
	\end{align}

	We begin by choosing $\nu_0=(1, 0, \dots, 0)$ and $d=0$. By the same calculations as in the $L^2$ case we then find that the quotient can be rewritten in the form
	\begin{align}
		\inf_{u\in C^{\infty}_0(\Pi_{\nu_0}^\limplus)} \frac{\int_{\Pi_{\nu_0}^\limplus} \bigl(|\grad'u|^2+4 \frac{\partial u}{\partial t}\langle \Lambda \xi', \grad'u\rangle+4|\xi'|^2 \bigl|\frac{\partial u}{\partial t}\bigr|^2\bigr)^{p/2}\, d\xi}
		{\int_{\Pi_{\nu_0}^\limplus} \frac{|u|^p}{x_1^p}\, d\xi}.
	\end{align}
	With the same ansatz as before, namely $u(\xi)=\phi(\xi') w(t)$, we can bound this from above by
	\begin{align}
		\inf_{\substack{\ \phi \in C^{\infty}_0(\R^{2n}_\limplus)\\ w \in C_0^{\infty}(\R)}}
		\frac{\int_{\R^{2n}_\limplus}\int_\R \bigl(|\grad\phi|^2|w|^2+4|w w'|\,|\phi\langle \Lambda \xi', \grad \phi\rangle|+4|\xi'|^2 |\phi|^2|w'|^2\bigr)^{p/2}\, dt\, d\xi'}
		{\int_{\R^{2n}_\limplus}\int_\R \frac{|\phi w|^p}{x_1^p}\, dt\, d\xi'}, 
	\end{align}
	where $\R^{2n}_\limplus= \{\xi'\in \R^{2n} : x_1>0\}$.

	In the case $p=2$ things are slightly simpler and the above quotient splits into three parts, one of which easily can be seen to be zero and another which can be eliminated by a simple scaling argument. However, in the general case we cannot in such a simple manner split the above integral. But using Jensen's inequality we can bound the quotient by some appropriately weighted sum of three terms and then use a similar scaling argument as for $p=2$.

	What we need is the following simple consequence of Jensen's inequality: For $\alpha\geq1$, $x_i\geq 0$ and any $a_i>0$, $i=1, \dots, k$, we have that
	\begin{align}
		\Bigl(\sum_{i=1}^k x_i \Bigr)^\alpha 
	&=
		\Bigl(\sum_{i=1}^k a_i\Bigr)^\alpha \biggl(\frac{\sum_{i=1}^k a_i (x_i/a_i)}{\sum_{i=1}^k a_i}\biggr)^\alpha\\
	&\leq
		\Bigl(\sum_{i=1}^k a_i\Bigr)^\alpha \biggl(\frac{\sum_{i=1}^k a_i (x_i/a_i)^\alpha}{\sum_{i=1}^k a_i}\biggr)\\
	&=
		\sum_{i=1}^k a_i^{1-\alpha}\Bigl(\sum_{j=1}^k a_j\Bigr)^{\alpha-1} x_i^\alpha.
	\end{align}
	We will apply this with $k=3$, $\alpha=p/2$, $x_1=|\grad \phi|^2|w|^2$, $x_2=4 |w w'|\,|\phi\langle \Lambda\xi', \grad \phi\rangle|$, $x_3=4 |\xi'|^2|\phi|^2|w'|^2$ and $a_i$'s to be chosen later. We also denote the effective weights of each $x_i$ by $c_i$, that is
	\begin{equation}
		c_i= a_i^{1-p/2} \Bigl( \sum_{j=1}^3 a_j\Bigr)^{p/2-1}.
	\end{equation}

	Using the above we find that 
	\begin{align}
	 \inf_{u\in C_0^\infty}& \frac{\int_{\Pi_{\nu_0}^\limplus} |\gradH u|^p\, d\xi}
		{\int_{\Pi_{\nu_0}^\limplus} \frac{|u|^p}{x_1^p}\, d\xi} 
	= \!\!
	\inf_{u\in C_0^\infty} \frac{\int_{\Pi_{\nu_0}^\limplus} \bigl(|\grad'u|^2+4 \frac{\partial u}{\partial t}\langle \Lambda \xi', \grad'u\rangle+4|\xi'|^2 \bigl|\frac{\partial u}{\partial t}\bigr|^2\bigr)^{p/2}\, d\xi}
		{\int_{\Pi_{\nu_0}^\limplus} \frac{|u|^p}{x_1^p}\, d\xi}\\[3pt]
	&\hspace{25pt}\leq \!\!
		\inf_{\substack{\phi \in C_0^\infty\\ w \in C_0^\infty}}\!
		\frac{\int_{\R^{2n}_\limplus}\int_\R \bigl(c_1|\grad\phi|^p|w|^p+2^p c_2|w w'|^{p/2}|\phi\langle \Lambda \xi', \grad \phi\rangle|^{p/2}+2^pc_3|\xi'|^p |\phi|^p|w'|^p \bigr) dt\, d\xi'}
		{\int_{\R^{2n}_\limplus}\int_\R \frac{|\phi w|^p}{x_1^p}\, dt\, d\xi'}\\[3pt]
	&\hspace{25pt}= \!\!
		\inf_{\substack{\phi \in C_0^\infty\\ w \in C_0^\infty }}\Biggl[
		c_1\frac{\int_{\R^{2n}_\limplus}|\grad\phi|^p\, d\xi'}{\int_{\R^{2n}_\limplus}\frac{|\phi|^p}{x_1^p}\, d\xi'} \cdot \frac{\int_\R |w|^p\, dt}{\int_\R |w|^p\, dt}
		+ 
		2^pc_2 \frac{\int_{\R^{2n}_\limplus}|\phi\langle \Lambda \xi', \grad \phi\rangle|^{p/2}\, d\xi'}{\int_{\R^{2n}_\limplus}\frac{|\phi|^p}{x_1^p}\, d\xi'} \cdot
		\frac{\int_\R |w w'|^{p/2}\, dt}{\int_\R |w|^p\, dt}\\[3pt]
		&\hspace{71pt}
		+2^pc_3
		\frac{\int_{\R^{2n}_\limplus}|\xi'|^p |\phi|^p\, d\xi'}{\int_{\R^{2n}_\limplus}\frac{|\phi|^p}{x_1^p}\, d\xi'}
		\cdot 
		\frac{\int_\R |w'|^p \, dt}{\int_\R|w|^p\, dt}\Biggr].
	\end{align}

	By a simple rescaling argument in $t$, by say replacing $w(t)$ by $\widetilde{w}(t)=w(\lambda t)$ with $\lambda>0$, one sees that we independently of choice of $\phi$ can make the last two terms arbitrarily small. Moreover, since the first term is independent of $w$ we can use the sharp Hardy inequality in $\R^{2n}_\limplus$ (see, for instance,~\cite{MR2777530, MR2124873}) to find that
	\begin{align}
	\inf_{u\in W^{1, p}_0}& \frac{\int_{\Pi_{\nu_0}^\limplus} |\gradH u|^p\, d\xi}
		{\int_{\Pi_{\nu_0}^\limplus} \frac{|u|^p}{x_1^p}\, d\xi} 
	\leq c_1\Bigl(\frac{p-1}{p}\Bigr)^p.
	\end{align}
	Since we have lost all dependence on $c_2$ and $c_3$ we are free to choose the $a_i$ in such a way that $c_1$ is arbitrarily close to $1$, which can be done by fixing $a_2, a_3$ positive and choosing $a_1$ sufficiently large. This completes the proof.
\end{proof}


\section{Hardy inequalities for convex domains in \texorpdfstring{$\Heis^n$}{Hn}}
\label{sec:IneqConvexSet}

We now turn our attention to using \thmref{thm:IneqHalfspace} to obtain a geometric version of Hardy's inequality on convex domains in $\Heis^n$. 

\begin{theorem}\label{thm:IneqConvexbodies2}
	Let $\Omega$ be a convex domain in $\Heis^n$. For $\xi\in \Omega$ let $\nu(\xi)$ denote the unit normal of $\partial\Omega$ at a point $\hat{\xi}\in\partial\Omega$, where $\hat\xi$ is such that $\dist(\xi, \partial\Omega)=\dist(\hat\xi, \xi)$. Then for any $u\in C^\infty_0(\Omega)$ we have that
	\begin{equation}
		\int_\Omega |\gradH u|^2\, d\xi\geq \frac{1}{4}\int_\Omega \sum_{i=1}^n \frac{\langle X_i(\xi), \nu(\xi)\rangle^2+\langle Y_i(\xi), \nu(\xi)\rangle^2}{\dist(\xi, \partial\Omega)^2}|u|^2\, d\xi.
	\end{equation}
\end{theorem}
\FloatBarrier

Note that we do not require the domain $\Omega$ to be bounded. In particular if $\Omega$ is a half-space of $\Heis^n$ this is precisely \thmref{thm:IneqHalfspace}. The proof is based on an approach used in~\cite{Avkhadiev1} and proceeds along the same lines as the that of \thmref{thm:IneqHalfspace} with an additional element in which we approximate the domain $\Omega$ by convex polytopes.

\begin{proof}[Proof of \thmref{thm:IneqConvexbodies2}]
	We begin by proving the inequality when $\Omega$ is a convex polytope. Let $\{F_k\}_k$ be the facets of $\Omega$ with corresponding inward pointing unit normals $\{\nu_k\}_k$. Further we construct a partition of $\Omega$ into the essentially disjoint sets $\Omega_k:=\{\xi\in\Omega : \dist(\xi, \partial\Omega)=\dist(\xi, F_k)\}$. Since the partition elements $\Omega_k$ are defined through a finite number of affine inequalities they are polytopes.

	For each partition element $\Omega_k$ we can now apply the same idea as in the proof of the previous theorem. The only difference will be that not all the boundary terms from the partial integration are zero. In each $\Omega_k$ we define the potential $V$ with components
	\begin{align}
		V_i(\xi) = \frac{\langle X_i(\xi), \nu_k\rangle}{\dist(\xi, F_k)},\\
		V_{n\limplus i}(\xi) = \frac{\langle Y_i(\xi), \nu_k\rangle}{\dist(\xi, F_k)}.
	\end{align}

	Through the same calculations as before one finds that
	\begin{align}\label{eq:PIOmegaK}
		0 & \leq
		\sum_{i=1}^n \int_{\Omega_k} \bigl(|(X_i+\alpha V_i)u|^2 +|(Y_i+\alpha V_{n\limplus i})u|^2\bigr) d\xi\\
		& =
		\sum_{i=1}^n \int_{\Omega_k} \bigl(|X_iu|^2+|Y_iu|^2-\alpha |u|^2 (X_i(V_i)+Y_i(V_{n\limplus i}))+\alpha^2 |u|^2(V_i^2+V_{n\limplus i}^2)\bigr) d\xi\\
		&\quad 
		 + \alpha\sum_{i=1}^n\int_{\partial \Omega_k} |u|^2 \bigl(V_i\langle X_i(\xi), \normal_k(\xi)\rangle+V_{n\limplus i}\langle Y_i(\xi), \normal_k(\xi)\rangle\bigr) d\Gamma_{\partial\Omega_k}(\xi),
	\end{align}
	where $\normal_k(\xi)$ denotes the outward pointing unit normal of $\partial\Omega_k$ at $\xi$. Note that on $F_k\subset \partial\Omega_k$ we have that $\normal_k(\xi)=-\nu_k$.

	Since $u$ is compactly supported in $\Omega$ the boundary contribution from $\partial \Omega$ is again zero, and thus all we need to deal with are the parts of $\partial \Omega_k$ that are in the interior of $\Omega$. For each such facet of $\Omega_k$ there is some $\Omega_l$, $l\neq k$, that shares this facet. Let $\Gamma_{kl}$ denote the common facet of $\Omega_k$ and $\Omega_l$, and note that $\normal_k{\mid_{\Gamma_{kl}}}=-\normal_l{\mid_{\Gamma_{kl}}}$. Summing over all partition elements $\Omega_k$ and letting $\normal_{kl}=\normal_{k}{\mid_{\Gamma_{kl}}}$, i.e.\ the unit normal of $\Gamma_{kl}$ pointing from $\Omega_k$ into $\Omega_l$, we obtain using the earlier calculations for the components of $V$ that

	\begin{align}\label{eq:SimplifyingBddry}
		0&\leq  
		\int_{\Omega} |\gradH u|^2\, d\xi - 
		\frac{1}{4}\int_\Omega \sum_{i=1}^n\frac{\langle X_i(\xi), \nu(\xi)\rangle^2+\langle Y_i(\xi), \nu(\xi)\rangle^2}{\dist(\xi, \partial\Omega)^2}|u|^2\, d\xi \\
		&\quad 
		-\frac{1}{2}\sum_{k\neq  l}\sum_{i=1}^n\int_{\Gamma_{kl}}  \frac{\langle X_i(\xi), \nu_k\rangle\langle X_i(\xi), \normal_{kl}\rangle+\langle Y_i(\xi), \nu_k\rangle\langle Y_i(\xi), \normal_{kl}\rangle}{\dist(\xi, F_k)}|u|^2\, d\Gamma_{kl}\\
		&=
		\int_{\Omega} |\gradH u|^2\, d\xi - 
		\frac{1}{4}\int_\Omega \sum_{i=1}^n\frac{\langle X_i(\xi), \nu(\xi)\rangle^2+\langle Y_i(\xi), \nu(\xi)\rangle^2}{\dist(\xi, \partial\Omega)^2}|u|^2\, d\xi \\
		&\quad 
		-\frac{1}{2}\sum_{k<l}\sum_{i=1}^n\int_{\Gamma_{kl}}  \frac{\langle X_i(\xi), \nu_k-\nu_l\rangle\langle X_i(\xi), \normal_{kl}\rangle+\langle Y_i(\xi), \nu_k-\nu_l\rangle\langle Y_i(\xi), \normal_{kl}\rangle}{\dist(\xi, F_k)}|u|^2\, d\Gamma_{kl}.
	\end{align}
	In the last equality we use the fact that $\Gamma_{kl}$ is by definition the set where $\dist(\xi, F_k)=\dist(\xi, F_l)$. 

	By construction we have that
	\begin{align}
			\Gamma_{kl}=\{\xi : \xi\cdot\nu_k-d_k=\xi\cdot\nu_l-d_l\}.
	\end{align}
	Rearranging we find that $\xi\cdot(\nu_k-\nu_l)-d_k+d_l=0$, that is, $\Gamma_{kl}$ is a hyperplane with normal $\nu_k-\nu_l$. Therefore, we have $\nu_k-\nu_l \parallel \normal_{kl}$ and all that remains to do is check that $(\nu_k-\nu_l)\cdot \normal_{kl}>0$. Since $\nu_k$ points into $k$-th partition element and $\normal_{kl}$ points out, $\nu_k\cdot \normal_{kl}$ is non-negative. By the same argument the term $\nu_l\cdot \normal_{kl}$ is non-positive. Therefore, the entire expression is non-negative and moreover we have that 
		\begin{align}
			|\nu_k-\nu_l|^2=(\nu_k-\nu_l)\cdot (\nu_k-\nu_l) &=
			2-2 \nu_k\cdot \nu_l \\
			&= 2-2\cos(\alpha_{kl}), 
		\end{align}
	where $\alpha_{kl}$ is the angle between $\nu_k$ and $\nu_l$. Thus $(\nu_k-\nu_l)\cdot \normal_{kl}=\sqrt{2-2\cos\alpha_{kl}}$ and equation~\eqref{eq:SimplifyingBddry} gives us that
	\begin{align}
		0 &\leq 
		\int_{\Omega} |\gradH u|^2\, d\xi - 
		\frac{1}{4}\int_\Omega \sum_{i=1}^n\frac{\langle X_i(\xi), \nu(\xi)\rangle^2+\langle Y_i(\xi), \nu(\xi)\rangle^2}{\dist(\xi, \partial\Omega)^2}|u|^2\, d\xi \\
		&\quad 
		-\frac{1}{\sqrt{2}}\sum_{k<l}\sum_{i=1}^n\int_{\Gamma_{kl}}\!\!\sqrt{1-\cos\alpha_{kl}}\, \frac{\langle X_i(\xi), \normal_{kl}\rangle^2+\langle Y_i(\xi), \normal_{kl}\rangle^2}{\dist(\xi, F_k)}|u|^2\, d\Gamma_{kl}.
	\end{align}
	We conclude that the boundary terms are of the correct sign, and the inequality follows. 

	Let now $\Omega$ be an arbitrary convex domain. For $u\in C_0^\infty(\Omega)$ we can always choose an increasing sequence of convex polytopes $\{\Omega_j\}_{j=1}^\infty$ such that $u\in C_0^\infty(\Omega_1)$, $\Omega_j\subset \Omega$ and $\Omega_j\to\Omega$ when $j\to\infty$. Letting $\nu_j(\xi)$ be the map $\nu$ from above corresponding to $\Omega_j$ we have that
	\begin{align}
		\int_\Omega |\gradH u|^2 \, d\xi 
		&= 
		\int_{\Omega_j} |\gradH u|^2 \, d\xi\\
		&\geq 
		\frac{1}{4} \int_{\Omega_j} \sum_{i=1}^n \frac{\langle X_i(\xi), \nu_j(\xi)\rangle^2+\langle Y_i(\xi), \nu_j(\xi)\rangle^2}{\dist(\xi, \partial\Omega_j)^2}|u|^2\, d\xi\\
		&= 
		\frac{1}{4} \int_{\Omega} \sum_{i=1}^n \frac{\langle X_i(\xi), \nu_j(\xi)\rangle^2+\langle Y_i(\xi), \nu_j(\xi)\rangle^2}{\dist(\xi, \partial\Omega_j)^2}|u|^2\, d\xi\\
		&\geq 
		\frac{1}{4} \int_{\Omega} \sum_{i=1}^n \frac{\langle X_i(\xi), \nu_j(\xi)\rangle^2+\langle Y_i(\xi), \nu_j(\xi)\rangle^2}{\dist(\xi, \partial\Omega)^2}|u|^2\, d\xi.
	\end{align}
	Letting $j$ tend to infinity completes the proof.
\end{proof}


\subsection{\texorpdfstring{$L^p$}{Lp} Hardy inequality for a convex domain}

Again we can make slight alterations to the proof above to obtain $L^p$-inequalities, $p\geq 2$, on convex domains. We summarize the results in the following theorem.

\begin{theorem}\label{thm:IneqConvexbodiesP}
	Let $\Omega$ be a convex domain in $\Heis^n$ and for $\xi\in \Omega$ let $\nu(\xi)$ denote the unit normal of $\partial\Omega$ at a point $\hat{\xi}\in\partial\Omega$, where $\hat\xi$ is such that $\dist(\xi, \partial\Omega)=\dist(\hat\xi, \xi)$. Then for any $p\geq 2$ and $u\in C_0^\infty(\Omega)$ the following inequality holds
	\begin{equation}\label{eq:IneqConvexBodiesP}
		\int_\Omega |\gradH u|^p\, d\xi \geq \Bigl(\frac{p-1}{p}\Bigr)^p \int_\Omega \sum_{i=1}^n \frac{|\langle X_i(\xi), \nu(\xi)\rangle|^p+|\langle Y_i(\xi), \nu(\xi)\rangle|^p}{\dist(\xi, \partial\Omega)^p}|u|^p\, d\xi.
	\end{equation}
\end{theorem}

\begin{proof}[Proof of \thmref{thm:IneqConvexbodiesP}]
	The proof of the theorem is very similar to the proof presented above for the $L^2$ case. We again begin with the case when $\Omega$ is a polytope and consider the same partition $\Omega_k$. For $g\in H^1(\Omega_k)$ and $V\in C^\infty(\Omega_k; \R^{2n+1})$ the divergence theorem gives us that
	\begin{align}\label{eq:DivTheoremLp}
	 	\int_{\Omega_k} \div(gV)|u|^p\, d\xi 
	 	& = 
	 	-p \int_{\Omega_k} g \langle V, \grad u\rangle \sgn(u) |u|^{p-1}\, d\xi+\int_{\partial \Omega_k} g  \langle V, \normal_k(\xi)\rangle |u|^p \, d\Gamma_{\partial\Omega_k}(\xi).
	\end{align} 
	Consider the first term in the right-hand side. Using H\"older's and Young's inequalities we have that
	\begin{align}
	 	-p \int_{\Omega_k} g \langle V, \grad u\rangle \sgn(u)|u|^{p-1}\, d\xi 
	 	&\leq 
	 	p \Bigl(\int_{\Omega_k}|\langle V, \grad u\rangle|^p\, d\xi\Bigr)^{1/p}\Bigl(\int_{\Omega_k}|g|^{p/(p-1)}|u|^p\, d\xi\Bigr)^{(p-1)/p}\\
	 	&\leq 
	 	\int_{\Omega_k}|\langle V, \grad u\rangle|^p\, d\xi+(p-1)\int_{\Omega_k}|g|^{p/(p-1)}|u|^p\, d\xi.
	\end{align} 
	Inserting this into \ref{eq:DivTheoremLp} and rearranging the terms we obtain
	\begin{align}
		\int_{\Omega_k}(\div(gV)-(p-1)|g|^{p/(p-1)})|u|^p\, d\xi
		& \leq
		\int_{\Omega_k}|\langle V, \grad u\rangle|^p\, d\xi\\
		&\quad  +
		\int_{\partial \Omega_k} g \langle V, \normal_k(\xi)\rangle|u|^p  \, d\Gamma_{\partial\Omega_k}(\xi).
	\end{align}
	Choosing $V=X_i$ and letting 
	\begin{equation}
		g=\alpha\, \sgn(\langle X_i(\xi), \nu_k\rangle) \biggl(\frac{|\langle X_i(\xi), \nu_k\rangle|}{\dist(\xi, F_k)}\biggr)^{p-1}
	\end{equation}
	we see by the same calculations as earlier that
	\begin{align}
		C(\alpha, p)\int_{\Omega_k}\frac{\langle X_i(\xi), \nu_k\rangle^p}{\dist(\xi, F_k)^p}&|u|^p\, d\xi
		\leq  
		\int_{\Omega_k}|X_iu|^p\, d\xi \\
	  &+
		\alpha\int_{\partial \Omega_k}\!\! \sgn(\langle X_i(\xi), \nu_k\rangle) \biggl(\frac{|\langle X_i(\xi), \nu_k\rangle|}{\dist(\xi, F_k)}\biggr)^{p-1}\!\!\langle X_i(\xi), \normal_k(\xi)\rangle |u|^p\, d\Gamma_{\partial\Omega_k}(\xi)
	\end{align}
	where 
	\begin{equation}
		C(\alpha, p)=-(p-1)(\alpha-|\alpha|^{p/(p-1)}).
	\end{equation}
	Maximizing this constant in $\alpha$ we find as before that
	\begin{equation}
	 	C(\alpha, p)\leq \Bigl(\frac{p-1}{p}\Bigr)^p, 
	\end{equation} 
	where the maximum is attained at
	\begin{equation}
		\alpha=-\Bigl(\frac{p-1}{p}\Bigr)^{p-1}.
	\end{equation}
	
	What remains to complete the proof is to show that we can discard the boundary terms after summing over the $\Omega_k$, in other words we need to prove that
	\begin{equation}
	 	-\Bigl(\frac{p-1}{p}\Bigr)^{p-1}\sum_{k}\int_{\partial \Omega_k}\!\! \sgn(\langle X_i(\xi), \nu_k\rangle) \biggl(\frac{|\langle X_i(\xi), \nu_k\rangle|}{\dist(\xi, F_k)}\biggr)^{p-1}\!\!\langle X_i(\xi), \normal_k(\xi)\rangle |u|^p\, d\Gamma_{\partial\Omega_k}(\xi) \leq 0.
	\end{equation}
	As in the $L^2$ case only the boundary terms that come from the interior of $\Omega$ are non-zero, and again these appear in pairs. Thus, in the same manner as before we wish to show that
	\begin{align}\label{eq:LpBoundaryTerms}
		\int_{\Gamma_{kl}} \biggl[\sgn(&\langle X_i(\xi), \nu_k\rangle) \biggl(\frac{|\langle X_i(\xi), \nu_k\rangle|}{\dist(\xi, F_k)}\biggr)^{p-1}\!\!\langle X_i(\xi), \normal_{kl}\rangle\\ 
		&- \sgn(\langle X_i(\xi), \nu_l\rangle) \biggl(\frac{|\langle X_i(\xi), \nu_l\rangle|}{\dist(\xi, F_l)}\biggr)^{p-1}\!\!\langle X_i(\xi), \normal_{kl}\rangle \biggr]|u|^p\, d\Gamma_{kl}
	\end{align}
	is non-negative.

	By the same geometric considerations as above $n_{kl}=\beta (\nu_k-\nu_l)$, where $\beta=1/{\sqrt{2-2\cos(\alpha_{kl})}}$. Using this combined with the fact that $\dist(\xi, F_k)=\dist(\xi, F_l)$ on $\Gamma_{kl}$ we find that~\eqref{eq:LpBoundaryTerms} can be rewritten as 
	\begin{align}
		\beta \int_{\Gamma_{kl}} \Bigl[&|\langle X_i(\xi), \nu_k\rangle|^{p}-\sgn(\langle X_i(\xi), \nu_k\rangle) |\langle X_i(\xi), \nu_k\rangle|^{p-1}\langle X_i(\xi), \nu_{l}\rangle\\ 
		&+ |\langle X_i(\xi), \nu_l\rangle|^{p}-\sgn(\langle X_i(\xi), \nu_l\rangle) |\langle X_i(\xi), \nu_l\rangle|^{p-1}\langle X_i(\xi), \nu_{k}\rangle \Bigr]\frac{|u|^p}{\dist(\xi, F_k)^{p-1}}\, d\Gamma_{kl}.
	\end{align}
	Thus it suffices to proof that the expression in the brackets is positive. Clearly this is positive when $\langle X_i(\xi), \nu_k\rangle$ and $\langle X_i(\xi), \nu_l\rangle$ have different signs. On the other hand if the two scalar products have the same sign we have an expression of the form $A^p -A^{p-1}B-B^{p-1}A+B^{p-1}$, where $A=|\langle X_i(\xi), \nu_k\rangle|$ and $B=|\langle X_i(\xi), \nu_l\rangle|$. But this we can rewrite as 
	\begin{equation}
		A^p -A^{p-1}B-B^{p-1}A+B^{p-1}=(A^{p-1}-B^{p-1})(A-B). 
	\end{equation}
	Now it is clear that both terms on the right-hand side have the same sign. Thus~\eqref{eq:LpBoundaryTerms} is non-negative and we conclude that the boundary terms can be discarded.

	By almost identical calculations we find the corresponding inequalities for the $Y_i$. Summing all terms and using Jensen's inequality as in the proof of \thmref{thm:IneqHalfspaceLP} we obtain the desired inequality for polytopes. The proof of the theorem can now be completed in the same manner as for \thmref{thm:IneqConvexbodies2}.
\end{proof}

\vspace{5pt}\noindent{\bf Acknowledgements.} It is a pleasure to thank Professor Ari Laptev for suggesting the problem studied here and for his input and encouragement. Part of this work was carried out during a visit to the Institute of Analysis, Dynamics and Modelling at Universit\"at Stuttgart and the author would like to thank the staff for their hospitality, and in particular Professor Timo Weidl for the invitation. The author would also like to thank Eric Larsson, Bartosch Ruszkowski and Aron Wennman for fruitful discussions. The author is supported by the Swedish Research Council grant no.\ 2012-3864.

\vspace{0.5cm}


\bibliographystyle{amsplain}
\bibliography{include/references} 

\providecommand{\bysame}{\leavevmode\hbox to3em{\hrulefill}\thinspace}
\providecommand{\MR}{\relax\ifhmode\unskip\space\fi MR }
\providecommand{\MRhref}[2]{%
  \href{http://www.ams.org/mathscinet-getitem?mr=#1}{#2}
}
\providecommand{\href}[2]{#2}
\begin{thebibliography}{10}

\bibitem{ArcozziFerrari1}
N.~Arcozzi and F.~Ferrari, \emph{{M}etric normal and distance function in the
  {H}eisenberg group}, Mathematische Zeitschrift \textbf{256} (2007), no.~3,
  661--684.

\bibitem{ArcozziFerrari2}
\bysame, \emph{The {H}essian of the distance from a surface in the {H}eisenberg
  group}, Ann. Acad. Sci. Fenn. Math. \textbf{33} (2008), no.~1, 35--63.

\bibitem{Avkhadiev1}
F.~G. Avkhadiev, \emph{{H}ardy type inequalities in higher dimensions with
  explicit estimate of constants}, Lobachevskii J. Math. \textbf{21} (2006),
  3--31.

\bibitem{BalEvLew}
A.~A. Balinsky, W.~D. Evans, and R.~T. Lewis, \emph{{The Analysis and Geometry
  of Hardy's Inequality}}, Universitext, Springer International Publishing,
  2015.

\bibitem{MR0001880}
W.~L. Chow, \emph{\"{U}ber {S}ysteme von linearen partiellen
  {D}ifferentialgleichungen erster {O}rdnung}, Math. Ann. \textbf{117} (1939),
  98--105.

\bibitem{Ambrosio}
L.~D'Ambrosio, \emph{{S}ome {H}ardy {I}nequalities on the {H}eisenberg
  {G}roup}, Differential Equations \textbf{40} (2004), no.~4, 552--564.

\bibitem{DanielliGarofaloPhuc}
D.~Danielli, N.~Garofalo, and N.~C. Phuc, \emph{Inequalities of
  {H}ardy-{S}obolev type in {C}arnot-{C}arath\'eodory spaces}, Sobolev spaces
  in mathematics. {I}, Int. Math. Ser. (N. Y.), vol.~8, Springer, New York,
  2009, pp.~117--151.

\bibitem{MR990239}
E.~B. Davies, \emph{Heat kernels and spectral theory}, Cambridge Tracts in
  Mathematics, vol.~92, Cambridge Univ. Press, Cambridge, 1989.

\bibitem{MR1070830}
N.~Garofalo and E.~Lanconelli, \emph{Frequency functions on the {H}eisenberg
  group, the uncertainty principle and unique continuation}, Ann. Inst. Fourier
  (Grenoble) \textbf{40} (1990), no.~2, 313--356.

\bibitem{Hormander}
L.~H\"ormander, \emph{{H}ypoelliptic second order differential equations}, Acta
  Mathematica \textbf{119} (1967), no.~1, 147--171.

\bibitem{MR2447488}
Y.~Jin, \emph{Hardy-type inequalities on {H}-type groups and anisotropic
  {H}eisenberg groups}, Chin. Ann. Math. Ser. B \textbf{29} (2008), no.~5,
  567--574.

\bibitem{LuanYang}
J.~Luan and Q.~Yang, \emph{{A} {H}ardy type inequality in the half-space on
  {$\R^n$} and {H}eisenberg group}, Journal of Mathematical Analysis and
  Applications \textbf{347} (2008), no.~2, 645--651.

\bibitem{MR2777530}
V.~Maz$'$ya, \emph{Sobolev spaces with applications to elliptic partial
  differential equations}, augmented ed., Fundamental Principles of
  Mathematical Sciences, vol. 342, Springer, Heidelberg, 2011.

\bibitem{PrandiRizziSeri}
D.~{Prandi}, L.~{Rizzi}, and M.~{Seri}, \emph{{A sub-Riemannian Santal\'o
  formula with applications to isoperimetric inequalities and first Dirichlet
  eigenvalue of hypoelliptic operators}}, ArXiv e-prints (2015), 1509.05415.

\bibitem{Rashevsky}
P.~K. Rashevsky, \emph{Any two points of a totally nonholonomic space may be
  connected by an admissible line}, Uch. Zap. Ped. Inst. Liebknechta, Ser.
  Phys. Math. \textbf{2} (1938), 83--94 (Russian).

\bibitem{RuszkowskiWeidl}
B.~{Ruszkowski} and T.~{Weidl}, \emph{{Hardy inequalities for the Heisenberg
  Laplacian on convex bounded polytopes}}, in preperation (2016).

\bibitem{RuzhanskySuragan}
M.~{Ruzhansky} and D.~{Suragan}, \emph{{Local Hardy inequality for sums of
  squares}}, ArXiv e-prints (2016), 1601.06157.

\bibitem{MR2124873}
J.~Tidblom, \emph{A {H}ardy inequality in the half-space}, J. Funct. Anal.
  \textbf{221} (2005), no.~2, 482--495.

\end{thebibliography}
\addcontentsline{toc}{chapter}{\numberline{}Bibliography} 

\end{document}